\documentclass[onefignum,onetabnum]{siamonline220329}

\usepackage{lipsum}
\usepackage{amsfonts}
\usepackage{graphicx}
\usepackage{epstopdf}
\usepackage{algorithmic}
\ifpdf
  \DeclareGraphicsExtensions{.eps,.pdf,.png,.jpg}
\else
  \DeclareGraphicsExtensions{.eps}
\fi

\newcommand{\setof}[1]{\left\{ {#1}\right\}}

\newcommand{\C}{{\mathbb{C}}}
\newcommand{\F}{{\mathbb{F}}}
\newcommand{\Q}{{\mathbb{Q}}}
\newcommand{\R}{{\mathbb{R}}}
\newcommand{\Z}{{\mathbb{Z}}}

\newcommand{\edit}[1]{\marginpar{\footnotesize{#1}}}

\def\lra{\longrightarrow}
\def\map#1{\,{\buildrel #1 \over \lra}\,}
\def\smap#1{{\buildrel #1 \over \to}}
\def\oo{\otimes}

\newcommand{\cond}{\mathfrak{c}}
\newcommand{\Hom}{\operatorname{Hom}}

\newcommand{\Tor}{\operatorname{Tor}}
\def\Pic{\operatorname{Pic}}

\newcommand\mycom[2]{\genfrac{}{}{0pt}{}{#1}{#2}}
\newcommand\barM{\overline M}
\newcommand\barN{\overline N}
\newcommand\barR{\overline R}

\usepackage{enumitem}
\setlist[enumerate]{leftmargin=.5in}
\setlist[itemize]{leftmargin=.5in}

\newsiamremark{example}{Example}
\newsiamremark{remark}{Remark}
\newsiamremark{hypothesis}{Hypothesis}
\crefname{hypothesis}{Hypothesis}{Hypotheses}
\newsiamthm{claim}{Claim}

\headers{Computing the Conley Index: a Cautionary Tale}{Konstantin Mischaikow and Charles Weibel}

\title{Computing the Conley Index: a Cautionary Tale\thanks{Submitted to the editors DATE.
\funding{K.M. was partially supported by the National Science Foundation under awards DMS-1839294 and HDR TRIPODS award CCF-1934924, DARPA contract HR0011-16-2-0033,  National Institutes of Health award R01 GM126555, and Air Force Office of Scientific Research under award number FA9550-23-1-0011. K.M. was also supported by a grant from the Simons Foundation. C. W. was supported by NSF grant 2001417.}}}

\author{Konstantin Mischaikow\thanks{Department of Mathematics, Rutgers University, New Brunswick, NJ
  (\email{mischaik@math.rutgers.edu}).}
\and Charles Weibel\thanks{Department of Mathematics, Rutgers University, New Brunswick, NJ
  (\email{weibel@math.rutgers.edu}).}
}

\usepackage{amsopn}

\ifpdf
\hypersetup{
  pdftitle={Computing the Conley Index: a Cautionary Tale},
  pdfauthor={Konstantin Mischaikow and Charles Weibel}
}
\fi

\begin{document}

\maketitle

\begin{abstract}

This paper concerns the computation and identification of the (homological) Conley index over the integers, in the context of  discrete dynamical systems generated by continuous maps. 
We discuss the significance with respect to nonlinear dynamics of using integer, as opposed to field, coefficients.
We translate the problem into the language of commutative ring theory.
More precisely, we relate shift equivalence in the category of finitely generated abelian groups to the classification of
$\Z[t]$-modules whose underlying abelian group is 
given.
We provide tools to handle the classification problem, but also highlight the associated computational challenges.
\end{abstract}


\section{Introduction}
This paper concerns the computation and identification of the (homological) Conley index in the context of a discrete dynamical system generated by a continuous map $f\colon X\to X$.

The Conley index \cite{conley:cbms, robbin:salamon, mrozek:90, szymczak, franks:richeson, mischaikow:mrozek} is a powerful algebraic topological invariant for the analysis of nonlinear dynamical systems for at least two reasons. 
First, it can be computed using finite data, and thus is applicable in the context of computational or data driven dynamics.
Second, there are a variety of theorems in which knowledge of the Conley index leads to information about the structure of the dynamics, e.g., existence of nontrivial invariant sets \cite{conley:cbms}, heteroclinic orbits \cite{conley:smoller}, 
fixed points \cite{srzednicki,mccord:88}, 
periodic orbits \cite{mccord:mischaikow:mrozek}, chaotic dynamics \cite{mischaikow:mrozek,szymczak,day:frongillo}, etc.

The computation of the Conley index begins with the identification of a pair of compact sets $P_0 \subset P_1$, called an \emph{index pair} \cite{robbin:salamon}, 
where for the sake of simplicity we assume that $f(P_i)\subset P_i$, $i=0,1$.  
The (homological) Conley index of the index pair is the shift equivalence class  of the \emph{index map}, i.e., the induced map on homology
\[
f_*\colon H_*(P_1,P_0;k) \to H_*(P_1,P_0;k).
\]
(See Section~\ref{sec:algebra} for the definition of shift equivalence.)
The index is important because, if $(P_1,P_0)$ and $(P'_1,P'_0)$ are index pairs with the property that the maximal invariant sets under $f$ in $P_1\setminus P_0$ and $P'_1\setminus P'_0$ are the same, then the associated Conley indices are the same.
The converse need not hold.

Computational identification of index pairs is relatively easy \cite{arai:kalies:kokubu:mischaikow:oka:pilarczyk,bush:gameiro:harker:kokubu:mischaikow:obayashi:pilarczyk,bush:cowan:harker:mischaikow}, 
but highly dependent upon the particular approximation used in the computation.
Therefore, in the context of applications, two related challenges appear.
First, to determine whether   $f_*\colon H_*(P_1,P_0; k) \to H_*(P_1,P_0;k)$ and $f_*\colon H_*(P'_1,P'_0;k) \to H_*(P'_1,P'_0;k)$  induce the same or different shift equivalence classes, and second, to  determine the shift equivalence class of $f_*$ with minimal computational effort.
We return to these challenges below.

Recall that $k[t]$ denotes the ring of formal polynomials with coefficients in $k$ (see \cite{Artin}).
The starting point for our analysis is the following observation: a $k$-module $A$ with an endomorphism $\alpha$ may be regarded as a $k[t]$-module, $M=(A,\alpha)$.  
Indeed, given a $k[t]$-module $M$, multiplication by $t$ is an endomorphism of the underlying $k$-module.
Conversely, given an endomorphism $\alpha$ of a $k$-module $A$, we obtain a $k[t]$-module structure on $A$ by letting $t$ act on $x\in A$
by $t\cdot x = \alpha(x)$.


Here is our module-theoretic interpretation of
shift equivalence; the proof is given in 
Section \ref{sec:algebra}.

\begin{proposition}
\label{prop:moduleiso}
Let $\alpha\colon A\to A$ and $\beta\colon B\to B$ be 
endomorphisms of finitely generated abelian groups, and let
$M=(A,\alpha)$ and $N=(B,\beta)$ be 
the associated $\Z[t]$-modules.
Then $\alpha$ and $\beta$ are shift equivalent (denoted by $\alpha \sim_s\beta$) if and only if 
$M\left[t^{-1}\right]\cong N\left[t^{-1}\right]$
as $\Z\left[t, t^{-1}\right]$-modules.
\end{proposition}


As a consequence, the issue of whether two index pairs 
have the same homological Conley index is decidable,
because it reduces to determining
whether $\Z[t,t^{-1}]$--modules are isomorphic; see \cite{BL}.

Returning to the  challenge of determining the shift equivalence class, it is often  computationally efficient to first compute
$f_*\colon H_*(P_1,P_0;k) \to H_*(P_1,P_0;k)$ when $k$ is a field.
In this case it is well known that
shift equivalence is completely determined by the rational canonical form of $f_*$, excluding nilpotent blocks, i.e., blocks with eigenvalue $t=0$.  
See \cite[7.3--7.5]{Lind}, for example.
An efficient rational canonical form algorithm is due to Storjohann \cite{storjohann} and implemented for Conley index computations in 
\cite{bush:cowan:harker:mischaikow}.

However, essential information can be lost if one considers shift equivalence over fields.

\begin{example}
\label{ex:1}
Consider two invariant sets  for  a one-dimensional 
map $f\colon \R\to \R$.
Let the first invariant set  consist of two unstable hyperbolic fixed points $\setof{x_0,x_1}$, e.g., 
$f(x_k)=x_k$ and $f'(x_k) = (-1)^k2$ for $k=1,2$.
Let the second invariant set consist of an unstable
orientation-preserving period two orbit $\setof{y_0,y_1}$ where  $f(y_0) = y_1$, $f(y_1) = y_0$, and $(f^2)' (y_k) = 2$.
Using the simplest possible index pairs 
(see \cite{mischaikow:mrozek}), the  
associated index maps on $H_1$ are
\begin{equation}
\label{eq:2Matrices}
    \begin{pmatrix}
    1 & 0 \\
    0 & -1
\end{pmatrix}
\quad\text{and}\quad
\begin{pmatrix}
    0 & 1 \\
    1 & 0
\end{pmatrix},
\end{equation}
respectively.
Since the eigenvalues for 
both these matrices 
are $\pm 1$, they are shift equivalent over any field.
However, a simple calculation  shows that 
they are not shift equivalent over $\Z$ 
(see Lemma~\ref{lem:invertible} and
Example \ref{ex:permute}
for a more general analysis).
This example shows that the Conley index can distinguish
between an invariant set consisting of two fixed points and a period two orbit -- clearly a result of interest in dynamical systems -- but at the cost of using integer coefficients.
\end{example}

This raises two questions:  how much information concerning the Conley index is lost by computing with field coefficients, and how difficult is it to compute with integer coefficients?
Complete answers to both questions appear to be extremely technical, and beyond the needs of current applications.
Thus, the focus of this paper is on providing the reader with hopefully useful insights on how integer computations could be done, and a sense of the algebraic challenges that need to be addressed to perform these computations. 

To perhaps further whet the reader's appetite, consider Example~\ref{ex:1} again. 
In Section~\ref{sec:invertible} we will show that every $2\times 2$ matrix with eigenvalues $\pm 1$ is shift equivalent to either $\left(\mycom10\mycom0{-1}\right)$ or $\left(\mycom01\mycom10\right)$, but not both.
Here is the significance of this result. Suppose that the induced map on homology of the index map is identified as having characteristic polynomial $x^n(x^2 -1)$. 
In this case, \eqref{eq:2Matrices} provides a complete identification of the homology Conley indices.
Unfortunately, as is made clear in this paper, this kind of  identification is difficult in general.

Here is an outline of this paper.
Section \ref{sec:algebra} provides a brief discussion the Conley index and explains why, 
for computational reasons, we restrict 
our attention to the homological Conley index.
The problem of shift equivalence is then 
translated into the realm of commutative ring theory 
and the proof of Proposition~\ref{prop:moduleiso} is presented.
In Section~\ref{sec:invertible} we provide an
elementary result in the setting that the 
endomorphisms are invertible, and apply it to the matrix algebra associated with Example~\ref{ex:1}.

The complexity of this result motivates our focus on the case $M=(\Z^2,T)$.
We use the form of the characteristic polynomial $\chi(t)=\det(t\cdot I-T)$ to organize our presentation.
In Section~\ref{sec:Z} we consider the case where $\chi(t)$ factors into linear terms.
If $\chi(t)$ is irreducible, then the problem of shift equivalence breaks up into two additional cases: $\Z[t]/(\chi)$ is a Dedekind domain, which is dealt with in Section~\ref{sec:Q}, and $\Z[t]/(\chi)$ is not a Dedekind domain, which is addressed in Section~\ref{sec:singular}.
In each of these sections, we provide a fundamental algebraic technique for identifying classes of shift equivalence, examples of how this technique can be employed, and a brief remark highlighting the technical difficulty of considering higher dimensional cases, i.e., $M=(\Z^n,T)$.

We conclude in Section \ref{sec:Z/pn} with a brief discussion of shift equivalence in the setting of finite abelian groups.

\section{Translation into Algebra} \label{sec:algebra} 

The goal of this section is the proof of 
Proposition~\ref{prop:moduleiso}, which states that 
the problem of identifying the shift equivalence class 
of a $\Z[t]$-module $M$ (represented by
an endomorphism $\alpha$ of the 
underlying abelian group) is equivalent to identifying the isomorphism class 
of the related module $M[t^{-1}]$.
We begin by reviewing the necessary concepts and notation.

\begin{definition}
\label{defn:shift}
In any fixed category,
endomorphisms $\alpha\colon A\to A$ and 
$\beta\colon B\to B$  are 
\emph{shift equivalent}, written 
$\alpha \sim_s \beta$, if there 
exist morphisms $r\colon A\to B$,  
$s\colon B\to A$, and a positive integer $m\in\Z^+$ such that
\begin{equation}
(i)\ r\circ \alpha = \beta\circ r,
\qquad
(ii)\
s\circ \beta = \alpha \circ s,
\qquad
(iii)\
s\circ r = \alpha^m,
\qquad \text{and} \qquad (iv)\ r\circ s = \beta^m.
\end{equation}
\end{definition}

\begin{example}\label{SE-matrix}
In the category of free $k$-modules, such as
vector spaces over a field $k$, endomorphisms
are represented by square matrices.
Square matrices $T_1$ and $T_2$ are 
shift equivalent if there are 
matrices $R$ and $S$ over $k$ such that 
$RT_1=T_2R$, $ST_2=T_1S$, $SR=T_1^m$ and $RS=T_2^m$.

It is well known that shift equivalence 
over a field $k$, such as $\Q$, is 
completely determined by the 
rational canonical form of $T$, excluding nilpotent 
blocks,  i.e., blocks with eigenvalue $t=0$. 
In particular, the nonzero eigenvalues are an invariant;
see \cite[7.3--7.5]{Lind}. 
This reflects the fact that finite-dimensional $k[T]$-modules are classified by
their rational canonical forms. Thus if $(M_1,T_1)$ and $(M_2,T_2)$
are shift equivalent over $k$, their characteristic polynomials
$\chi(t)=\det(t\cdot I-T_i)$ differ only by powers of $t$ 
and the $T_i$ have the same rational canonical form. 
An efficient rational canonical form algorithm is due to Storjohann \cite{storjohann}.
\end{example}

\paragraph{Homotopy theory}
The combined work of \cite{szymczak, franks:richeson} shows that the most 
general form of the Conley index is shift equivalence in the homotopy category of 
maps on pointed topological spaces. 
This implies that shift equivalence
of homotopy groups 
(in the category of groups)
is an invariant of the Conley index.
Thus, in this general setting the issue of whether two index pairs have the same Conley index requires the ability to decide if two finitely generated groups are isomorphic.
This is known to be impossible; see \cite[7.10]{Artin}. 
Therefore, from  the perspective of applications, working on the level of the homotopy Conley index is not a natural starting point.
With this in mind, we focus on the homological Conley index. 
Consequently, we are interested in shift equivalence in the category of finitely generated abelian groups.

Stated more explicitly, let $(P_1,P_0)$ and $(Q_1,Q_0)$ be index pairs for continuous maps $f$ and $g$ (it is possible that $f=g$).
We are interested in understanding whether $f_*\colon H_*(P_1,P_0;k)\to H_*(P_1,P_0;k)$ and $g_*\colon H_*(Q_1,Q_0;k)\to H_*(Q_1,Q_0;k)$ are shift equivalent or not. 
We leave it to the reader to check that 
$f_*$ 
and $g_*$ 
are shift equivalent if and only if $f_n\colon H_n(P_1,P_0;k)\to H_n(P_1,P_0;k)$ and $g_n\colon H_n(Q_1,Q_0;k)\to H_n(Q_1,Q_0;k)$ are shift equivalent for each $n$.

We finish our discussion of the Conley index by citing a result of J. Bush \cite[Corollary 4.7]{bush} that every $n\times n$ matrix $T$ with integer entries can be realized as a representative of a Conley index.
More precisely, given $T$ there exists a one-dimensional continuous function $f$ and an index pair $(P_1,P_0)$ such that $T$ is shift equivalent over $\Z$ to $f_1\colon H_1(P_1,P_0)\to H_1(P_1,P_0)$.

Turning to the algebraic formulation of shift equivalence, recall
\cite{AM} that the \emph{localization} $M[t^{-1}]$ of a $k[t]$-module $M$ is the set of equivalence classes of formal fractions $x/t^i$, where $x\in M$, $i\ge0$, and $x/t^i\equiv y/t^j$ if and only if $t^{j+m}x=t^{i+m}y$ for some $m>0$.

\begin{proof}[Proof of Proposition \ref{prop:moduleiso}]
Assume $\alpha \sim_s\beta$,
and let $r:A\to B$, $s:B\map{}A$ and $m$
be as in Definition \ref{defn:shift}.
Then $r$ is a $k[t]$-module homomorphism
from $M=(A,\alpha)$ to $N=(B,\beta)$,
because for all $x\in A$:
\[
r(t\cdot x) = r(\alpha(x)) = \beta(r(x)) = t\cdot r(x).
\]
The same argument shows that $s$ is a $k[t]$-module homomorphism.
The conditions that $sr=t^m$ 
and $rs=t^m$ translate into
$sr(x) = \alpha^m(s(x)) = t^m\cdot x$ and
$rs(y) = \beta^m(r(y)) = t^m\cdot y$.
Passing to $M[t^{-1}]$
and $N[t^{-1}]$, this is
equivalent to $t^{-m}\cdot s(r(x)) = x$
and $r\cdot(t^{-m}\cdot s) y=y$. Therefore
$t^{-m}\cdot s$ is an inverse of $r$ and
$t^{-m}\cdot r$ is an inverse of $s$.
Thus $M[t^{-1}]\cong N[t^{-1}]$.

Now assume that there exists a 
$k\!\left[t, t^{-1}\right]$--module isomorphism 
$f\colon M\left[ t^{-1}\right] \to 
N\left[ t^{-1}\right]$.
Because $M$ is finitely generated,
say by $x_1,...,x_n$, there are $d_i>0$
and $y_i\in N$ such that $f(x_i)=y_i/t^{d_i}$.
Let $d = \max \setof{d_1,\ldots, d_n}$.
Set $r(x) = t^d f(x)$ and observe that $r:M\to N$ is a group homomorphism and $f(x)=r(x)/t^d$. 
Similarly, the isomorphism $f^{-1}:N\left[t^{-1}\right] \to 
M\left[t^{-1}\right]$ has the form $f^{-1}(y)=s(x)/t^e$ for some $e>0$. 
Then for all $x\in M$ we have $x=f^{-1}f(x) = r(s(x))/t^{d+e}$, i.e., 
$r(s(x))=t^{d+e}x$; similarly we have $s(r(y))=t^{d+e}y$ for all $y\in N$.
Thus $\alpha$ and $\beta$ are shift equivalent.
\end{proof}

\begin{remark}
The {\it Bowen--Franks group} of $M$ is
$M/(1-t)M$; see \cite[7.4.15]{Lind}.
Since this is a quotient of $M[t^{-1}]$,
this invariant is weaker than the 
invariant we consider.
\end{remark}

As we pointed out in the introduction, 
most Conley index computations are done with $k$ chosen to be a field 
using rational canonical forms,
for the sake  of computational efficacy (see \cite[Section 14.8]{Artin}). 
Even though the worst bounds on computational complex of 
homology computations with integer coefficients 
are worse than that of fields, computations over the integers are possible.
Thus for the remainder of this paper we assume that  $k\cong \Z$, and $M$ is a $\Z[t]$-module, finitely generated 
as an abelian group, with $t$ acting as an endomorphism of the underlying abelian group.

\begin{remark}
For the sake of simplicity, 
we will talk about the shift 
equivalence class of a $\Z[t]$-module
$M$, meaning the shift equivalence class
of the map $M\to M$, $m\mapsto tm$. 
We will say that
a $\Z[t]$-module $M$ is finitely generated 
if it is finitely generated as an abelian group; 
and that $M$ is torsionfree if it is
torsionfree as an abelian group.
\end{remark}

We focus first on $\Z[t]$-modules
$M$ which are finitely generated and torsionfree as abelian groups.
That is, the underlying abelian group is $\Z^m$ and $t$ acts by an 
$m\times m$ integer matrix $T$. 
As in Example \ref{SE-matrix},
the characteristic polynomial $\chi_M(t)=\det(t\cdot I-T)$
is an invariant in $\Z[t]$ 
up to powers of $t$.
The following result allows us to assume that a
torsionfree $\Z[t]$-module $M$ has no $t$-torsion.

Set $M_{\mathrm{nil}}=\{x\in M: t^nx=0, n\gg0\}$.
Then $M/M_{\mathrm{nil}}$ is also a $\Z[t]$-module.

\begin{lemma}\label{lem:no-t} 
If $M$ is a $\Z[t]$-module,
finitely generated and
torsionfree as an abelian group,
then $M/M_{\mathrm{nil}}$ is torsionfree
and $M[t^{-1}]\smap{\cong}
M/M_{\mathrm{nil}}[t^{-1}]$.
\end{lemma}

\begin{remark}
Proposition \ref{prop:moduleiso} and
Lemma \ref{lem:no-t} imply that
$(M,t)$ is shift equivalent to
$(M/M_{\mathrm{nil}},t)$.
\end{remark}

\begin{proof}
If $x\in M$ and  there exists $a\in\Z$ such that 
$ax\in M_{\mathrm{nil}}$, 
then $t^n(ax)=a(t^nx)=0$.
Since $M$ is torsionfree this implies that $t^nx=0$ 
and hence $x\in M_{\mathrm{nil}}$. 
This implies that if 
$x\in M/M_{\mathrm{nil}}$ then $ax\ne0$
for all $a\ne0$, i.e.,
$M/M_{\mathrm{nil}}$ is
torsionfree as an abelian group.

Finally, since
$M_{\mathrm{nil}}$ 
is finitely generated there is an $m$
such that $t^m\cdot M_{\mathrm{nil}}=0$, and hence
the map $s:M\to M$, $s(x)=t^mx$,
factors through a map $S:M/M_{\mathrm{nil}}\to M$
with $S\circ q=t^m$, where $q$ is the quotient map
$q:M\to M/M_{\mathrm{nil}}$.
(See \cite[14.1.6]{Artin}.)
Thus $q$ and $S$ form a shift equivalence
between $M$ and $M/M_{\mathrm{nil}}$.
\end{proof}

\begin{remark}\label{det(T)ne0}
$M_{\mathrm{nil}}$ is zero if and only if the
determinant of the associated matrix $T$ is nonzero. 
\end{remark}

\begin{remark}
As with any finitely generated $\Z[t]$-module, 
$M$ has associated prime ideals $\wp_i$ in 
$\Z[t]$ and submodules $Q_i$ of $M$ such that
$0=Q_0\cap\cdots\cap Q_n$.
See \cite[4.20--22]{AM}. 
In this primary decomposition,
the $Q_i$ are associated to $\wp_i$ in the sense that
\[
\wp_i = \{f\in \Z[t]:f^n\cdot M\subset Q_i
\text{ for }n\gg0\}.
\]
\end{remark}

\begin{definition}\label{def:I}
Let $M$ be a $\Z[t]$-module whose underlying abelian group is $\Z^n$.
Throughout this paper we set $R=\Z[t]/I$, where
 $I$ is the ideal 
$\{f\in\Z[t]\colon 
f(x)=0\textrm{ on } M\}$
of $\Z[t]$.
\end{definition}

\begin{lemma}
Let $M$ and $I$ be as in Definition \ref{def:I}.
Then, $M$ is an $R$-module, 
and $I$ is a principal ideal
of $\Z[t]$, generated by a monic polynomial.
\end{lemma}

\begin{proof}
We adopt standard terminology; see  \cite{Artin, AM}.
Since $I\cap\Z=0$ and $I$ contains the monic polynomial $\chi(t)$,
every associated prime of $M$ is generated by a monic polynomial.
Since $\Z[t]$ is a unique factorization domain, 
we can factor $\chi(t)$ as a product of irreducible
polynomials, and these must be monic. 
Let $h(t)$ denote the
minimal polynomial in $\Q[t]$
of $t$ acting on $M$. Then
$h(t)$ is monic and divides $\chi(t)$; 
clearing fractions, we may assume $h$ is a primitive polynomial in $\Z[t]$,
i.e., its coefficients are relatively prime integers.
Since $h$ divides $\chi(t)$ in $\Q[t]$, 
there is a $g(t)$ in $\Z[t]$ and a constant $c$ such that 
$h(t)g(t)=c\chi(t)$;
$c$ must be the 
greatest common divisor 
of its coefficients, i.e.,
the content of $g$.
Replacing $g$ by $g/c$ we have $hg=\chi$.
This implies that $h(t)$ is a product of monics and hence is monic in $\Z[t]$.
\end{proof}

Conversely, any $R$-module may be 
considered as a $\Z[t]$-module by letting $f\in\Z[t]$
act as its image in $R=\Z[t]/I$ acts; this change from
$R$-modules to $\Z[t]$-modules is called 
{\it restriction of scalars} \cite{AM}.
Thus $R$-modules $M$ and $N$ 
are shift equivalent if and only if 
$M[t^{-1}]\cong N[t^{-1}]$ 
as $R[t^{-1}]$-modules.

\begin{remark}
When $\chi(t)$ is an irreducible 
polynomial $f$ of degree~2, $M$ is a 
module over the 1--dimensional domain 
$R=\Z[t]/I$, and the field of
fractions of $R$ is a
number field. 
This case is discussed in Sections~\ref{sec:singular} and \ref{sec:Q}.
\end{remark}

\goodbreak
 \section{Invertible matrices}
 \label{sec:invertible}

It is well known that conjugate matrices 
are shift equivalent: $T$ is shift equivalent 
to $RTR^{-1}$ via $R$ and $S=R^{-1}$. 
Here is a partial converse. Recall that a matrix
$T$ over the integers is invertible if and
only if $\det(T)=\pm1$. 
 
\begin{lemma}\label{lem:invertible}
Suppose that $T_1$ is shift equivalent to $T_2$ (via $R$ and $S$).
If  $T_1$ is invertible and  $det(T_2)\ne0$, then $T_2$, $R$ and $S$ are invertible and $T_2=RT_1R^{-1}$. 
\end{lemma}

\begin{proof}
The axiom that $SR=T_1^m$ 
implies that 
$R:A\to B$ is an injection and $S:B\to A$ is a surjection.
Therefore, $B\cong R(A)\oplus \ker(S)$ (see \cite[Theorem IV.1.18]{hungerford}).
Because $\det(T_2)\ne0$, the axiom that $RS=T_2^m$ implies that $\ker(S)=0$.
Hence $R$ and $S$ are invertible and $S=T_1^mR^{-1}$.
The axiom that $RT_1=T_2R$ implies
that $T_2=RT_1R^{-1}$.
\end{proof}

We now present a sequence of examples that are consequences of Lemma~\ref{lem:invertible}; they are indicative of the types of results obtained in the more challenging settings discussed in the sections that follow.

The only simple general result that we are aware of is that the $n\times n$ matrices $\pm I_n$ are not shift equivalent to any other $n\times n$ matrix because they are in the center of $GL_n(\Z)$. (This follows from 
Lemma \ref{lem:invertible}.)

\begin{example}\label{ex:permute}
Returning to Example~\ref{ex:1}, we claim that $P=\left(\mycom0{1}\mycom{1}0\right)$ is shift equivalent to the matrix $\left(\mycom1x\mycom0{-1}\right)$ if and only if $x$ is odd.
To see this, conjugate $P$ with $R=\left(\mycom ac\mycom bd\right)$ (where $\det(R) =  \pm 1$) to get
\[
\begin{bmatrix}
bd-ac & a^2-b^2 \\
d^2-c^2 & ac -bd 
\end{bmatrix}
=
\begin{bmatrix}
1 & 0 \\
x & -1 
\end{bmatrix}.
\]
Solving gives $a=\pm b$,
$a(c \pm d)=\pm1$, $a=\pm1$ and
$x = 1 \pm 2c$, 
so $x$ is odd.
$P$ is also shift equivalent
to $-P = \left(\mycom0{-1}\mycom{-1}0\right)$ 
via $R=\left(\mycom{-1}0\mycom0{1}\right)$.
\goodbreak

A similar argument shows that 
$Q=\left(\mycom10\mycom0{-1}\right)$
is shift equivalent to the matrix 
$\left(\mycom1c\mycom0{-1}\right)$ if and only if $c$ is even,
and that $Q\sim_s -Q$.
\end{example}

\begin{example}\label{ex:rotate}
Suppose that $\chi(t)=t^2+1$. 
Then the rotation matrix
$T=\left(\mycom0{-1}\mycom{1}0\right)$
is shift equivalent to every matrix
of the form
$\left(\mycom c{1+c^2}\mycom{-1}{-c}\right)$, via
$R=\left(\mycom 1{-c}\mycom 0{-1}\right)$. 
In particular, $T\sim_s -T$.
Similarly, $T$ is shift equivalent to
every matrix of the form
$\left(\mycom c{1-c^2}\mycom{+1}{-c}\right)$
and $\left(\mycom{c}{-1} \mycom{1+c^2}{-c}\right)$.
\end{example}

\begin{proposition}
\label{prop:2perm}
Every integer matrix $T$ with
$\chi(t)=t^2-1$ is
shift equivalent to either
$P=\left(\mycom0{1}\mycom{1}0\right)$ or $Q=\left(\mycom10\mycom0{-1}\right)$.
\end{proposition}

An alternate proof is given by Example \ref{groupring} below.

\begin{proof}
Let $T(x,u,v)$ denote the matrix
$\left(\mycom xv\mycom u{-x}\right)$
with $x^2+uv=1$ and
$\chi(t)=t^2-1$.

We proceed by induction on $|x|$.
When $x=0$, we get the matrices $P$ and $-P$ of Example \ref{ex:permute}.
When $|x|=1$, we get the triangular matrices of Example \ref{ex:permute},
which are shift equivalent to either $P$ or $Q$.

Suppose that $|x|\ge2$.
Since $x^2-1=-uv$ either $|u|$ or $|v|$ is less than $|x|$
but not both, and $u$ and $v$ have opposite signs.
Conjugating with 
$E=\left(\mycom11\mycom01\right)$ and $E^{-1}$ yields 
$$
ETE^{-1} = T(x-u,u,v-u+2x); \quad
E^{-1}TE = T(x+u,u,v+u-2x).
$$
If $|u|<|x|$, either $|x-u|<|x|$ or $|x+u|<|x|$ and we are done.
Similarly, if $|v|<|x|$, conjugating $T$ with 
$\left(\mycom10\mycom{-1}1\right)$
(resp., its inverse)
yields $T(x-v,u-v+2x,v)$, respectively, $T(x+v,u+v+2x,v)$, and we are done in this case as well.
\end{proof}

A similar analysis using $T(x,u,v)$
with $uv=1+x^2$ shows that every integer matrix with $\chi(t)=t^2+1$ is 
shift equivalent to the rotation matrix
$T=\left(\mycom0{-1}\mycom{1}0\right)$.
A different proof is given in 
Example \ref{d-ne-1}(1) below.

\begin{remark}
Extending Example~\ref{ex:1} to a periodic $n$ orbit gives rise to an index map whose associated characteristic polynomial has the form $t^n-1$.
As is indicated in Example~\ref{ex:3perm}, identifying the associated shift equivalence classes over $\Z$ is non-trivial.
\end{remark}

\section{Shift equivalence when $\chi(t)$ factors into linear terms}
\label{sec:Z}

As indicated in the introduction, we shall focus for simplicity on shift equivalence between $2\times2$ matrices over $\Z$. 
First, we handle the easy case, when the
characteristic polynomial $\chi(t)$ factors in $\Z[t]$, i.e., $\chi(t)=(t-\lambda_1)(t-\lambda_2)$,
and $T$ is a lower-triangular matrix.

For $a\in\Z$, we write $M_a$ for the $\Z[t]$-module which is the abelian group $\Z^2$ with
$T=\bigl(\mycom{\lambda_1}a
\mycom0{\lambda_2}\bigr)$,
i.e., $t$ acts by 
$t(x,y)=(\lambda_1 x,\lambda_2 y+a x)$.
Note that $M_a$ is conjugate to both $M_{-a}$ and $(\Z^2,T')$, with $T'=\bigl(\mycom{\lambda_2}0 \mycom a{\lambda_1}\bigr)$.
Therefore, $T$ is shift equivalent to $\bigl(\mycom{\lambda_1}{-a}
\mycom0{\lambda_2}\bigr)$ and $T'$.

In general, a $\Z[t]$-module map 
$h\colon M_a\to M_b$ may 
be represented as a map 
$\Z^2 \to \Z^2$ given by a lower triangular matrix
$R=\bigl(\mycom ru\mycom0s\bigr)$
such that 
\[
\begin{pmatrix}
\lambda_1 & 0 \\  b & \lambda_2
\end{pmatrix}
\begin{pmatrix}
r & 0 \\  u & s
\end{pmatrix}
 = 
\begin{pmatrix}
r & 0 \\  u & s
\end{pmatrix}
\begin{pmatrix}
\lambda_1 & 0 \\  a & \lambda_2
\end{pmatrix},
\]
i.e,
\begin{equation}
\label{eq:fmap}
\begin{pmatrix}
\lambda_1 r & 0 \\  br + \lambda_2 u & \lambda_2s
\end{pmatrix}
 = 
\begin{pmatrix}
r\lambda_1 & 0 \\ as + u\lambda_1 & s\lambda_2
\end{pmatrix}.
\end{equation}

Recall from Proposition~\ref{prop:moduleiso} that  
$M_a$ is shift equivalent to $M_b$ if and only if
$M_a[t^{-1}]$ is isomorphic to $M_b[t^{-1}]$.
We spend the rest of this section identifying conditions under which $
h[t^{-1}]\colon M_a[t^{-1}]\to M_b[t^{-1}]$ provides such an isomorphism. 

 We first consider the case when
$\lambda_1=\lambda_2$, i.e., 
when $T$ has just one 
Jordan block.

\begin{proposition}\label{one-lambda}
The shift equivalence classes of $T_a=\bigl(\mycom{\lambda}a
\mycom0{\lambda}\bigr)$, 
$\lambda \neq 0$, are in 1--1 
correspondence with the infinite set of positive integers $a$ such that $a$ is relatively prime to $\lambda$.
\end{proposition}

\begin{proof}
When $\lambda=\lambda_1=\lambda_2$, the condition that $h$ be a
module map is that $as=br$.
Now $h$ induces an isomorphism $M_a[t^{-1}]\cong M_b[t^{-1}]$ if and only if $\det(h)=rs$ is a unit in $\Z[\lambda^{-1}]$, i.e.,  
if and only if $r$ and $s$ divide $\lambda^n$ for some $n$.
Therefore $M_a$ and $M_b$ are shift equivalent 
if and only if $as=br$, where 
$r$ and $s$ are integers which 
become units in
$\Z[\lambda^{-1}]$.
\end{proof}

\begin{example}\label{ex-lambda}
If $b$ divides $\lambda^n$, the map $M_1\map{h}M_b$,
$h(x,y)=(x,by)$ induces a shift equivalence.
More generally, if $b=as$ and $s$ divides $\lambda^n$, the map 
$h\colon M_a\to M_b$, $(x,y)\mapsto(x,sy)$, is 
part of a shift equivalence.
\end{example}

\begin{proposition}\label{Ma=Mb}
The $\Z[t]$-modules $M_a$ and $M_b$ are shift equivalent if and only if there are integers 
$r$,$s$ with the same prime factors as $\lambda_1\lambda_2$
such that $as-br$ is divisible
by $(\lambda_1-\lambda_2)$.
\end{proposition}

\begin{proof}
From the matrix equality 
\eqref{eq:fmap} before
Proposition \ref{one-lambda},
we see that a necessary and
sufficient condition is that
$as-br = u(\lambda_2-\lambda_1)$,
and  $\det(h)=rs$ is a unit in 
$\Z[\lambda_1^{-1},\lambda_2^{-1}]$.
\end{proof}

\begin{example}
If $|\lambda_1-\lambda_2|=1$ then 
every $M_a$ is shift equivalent to $M_0$, 
because the condition in
Proposition \ref{Ma=Mb} is 
satisfied for all $a,b$.

If $|\lambda_1-\lambda_2|=2$, either
both $\lambda_i$ are even,
in which case every $M_a$ is 
shift equivalent to $M_0$, or else
both $\lambda_i$ are odd, in which case
there are two shift equivalence classes:
$M_a$ with $a$ even, and
$M_a$ with $a$ odd.
\end{example}

\begin{example}
If $|\lambda_1-\lambda_2|=p$ 
is an odd prime, and
$\lambda_1$ and $\lambda_2$ are
prime to $p$, the issue is whether
the primes dividing 
$\lambda_1\lambda_2$ generate the
cyclic group of units of $\Z/p$.
In any event, $M_a$ is not shift
equivalent to $M_0$ because
$p$ does not divide $\lambda_1$ or $\lambda_2$.

For example, if $p=17$ then the units
of $\Z/17$
are cyclic of order 16, generated by
6 with $6^2\equiv2\pmod{17}$. 
If $(\lambda_1,\lambda_2)=(2,19)$
then there are 4 shift equivalence classes of $M_a$ $(a=0,\pm6,2).$
If $(\lambda_1,\lambda_2)$ 
is $(1,18)$ or $(3,20)$
then there are 2 shift equivalence classes of $M_a$ (for $a=0,1$).
\end{example}

\begin{example}
If $(\lambda_1,\lambda_2)=(1,p)$
with $p$ prime,
then $M_a\sim_s M_b$ if and only if
$a\equiv\pm b \mod(p-1)$.
Thus if $p$ is odd
there are $(p-1)/2$
shift equivalence classes;
if $p=2$ there is only one 
shift equivalence class.

Similarly, if $(\lambda_1,\lambda_2)=(1,p^n)$
then $M_a\sim_s M_b$ if and only if 
$a\equiv\pm p^kb \mod(p^n-1)$
for some $k<n$.
\end{example}

If $\lambda_1$ is relatively prime to
$\lambda_2$, then the diagonal matrix
$M_0$ is not shift equivalent to $M_a$ for any nonzero integer $a$.
Indeed, $as\not\equiv0$ modulo
$\lambda_1 - \lambda_2$.

\begin{remark}\label{SE-R}
Proposition~\ref{Ma=Mb} can be generalized to any commutative ring $R$.
In particular, given $\lambda_1, \lambda_2\in R$, let $M^R_a$ denote the $R[t]$-module which is $R^2$ as an $R$-module, with $t$ acting by $t(x,y)=(\lambda_1 x,\lambda_2 y+a x)$.
Then, the proof of Proposition~\ref{Ma=Mb} goes through to show
that $M^R_a$ and $M^R_b$ are shift equivalent if and only if
$a$ and $b$ differ by a unit of 
$R[\lambda_1^{-1},\lambda_2^{-1}]$,
modulo $(\lambda_1-\lambda_2)$.
This will be used with $R=\Z/p^n$ 
and $\lambda_1=\lambda_2$ in 
Section \ref{sec:Z/pn}.
\end{remark}

When $M=\Z\oplus\Z/m$, every $T:M\to M$ has the form
$\bigl(\mycom{\lambda_1} a\mycom0{\lambda_2}\bigr)$
for $\lambda_1\in\Z$ and $a,\lambda_2\in\Z/m$.
Passing to $M\otimes\Q$ and $M/mM$, we see that
$\lambda_1$ and $\lambda_2$ are 
shift equivalence invariants. 
We write $M_a$ for this
$\Z[t]$-module, and $\bar\lambda_1$
for the image of $\lambda_1$ in $\Z/m$.
Note that $r$ is a unit in $\Z[\lambda_1^{-1}]$ if and only if
 $r\in\Z$ has the same prime factors as $\lambda_1$.
Using \eqref{eq:fmap}, 
the proof of \ref{Ma=Mb} 
goes through to show:

\begin{corollary}\label{Z+Z/m}
When $M=\Z\oplus\Z/m$, and $\lambda_1\in\Z$,
$\lambda_2\in\Z/m$ are nonzero,
then:
\begin{enumerate}
    \item 
$M_a$ and $M_b$ are 
shift equivalent if and only if
there is an $r\in\Z$ with the same prime factors as $\lambda_1$, and an $s\in\Z/m$ with the same prime 
factors as $\lambda_2$ so that
$as\equiv br$ modulo $\bar\lambda_1-\lambda_2$.
\item If $\lambda_2\equiv\lambda_1\pmod m$ and
$\lambda_1$ is relatively prime to $m$, then:\\
$M_a$ and $M_b$ are shift equivalent if and only if 
$a$ and $b$ differ by a unit of 
$\Z[\lambda_1^{-1}]/m$.
\item In particular, if $m$ is prime then 
shift equivalence classes on $M=\Z\oplus\Z/m$ 
are completely classified by $\lambda_1\in\Z$ and $\lambda_2\in\Z/m$.
\end{enumerate}
\end{corollary}

\begin{remark}
\label{rem:xxx}
Our discussion in this section has focused on shift equivalence between $2 \times 2$ matrices over $\Z$ where the characteristic polynomial $\chi(t)$ factors into linear terms. 
Using similar arguments,
one could analyze the general case 
where $2$ Jordan blocks are replaced 
by $n$ Jordan blocks. However, 
the complexity of determining the 
shift equivalence classes grows rapidly.
 Determining $h$ requires satisfying a system of $n(n-1)/2$ Diophantine equations arising from the
 analogue of \eqref{eq:fmap}.
For individual examples these computations can be done, but we do not know of a simple closed form expression for the number of shift equivalence classes based on the eigenvalues of $T$.
\end{remark}

\section{\bf Integers in quadratic number fields}\label{sec:Q}

Still assuming $T$ is a $2\times2$ matrix,
we now examine the case where the characteristic polynomial $\chi(T)$
is irreducible.
This implies that $R=\Z[t]/(\chi)$ 
is a 1-dimensional integral domain, isomorphic to $\Z^2$ as an abelian group
\cite[Chapter 15]{Artin}.
Let $\xi$ denote the image of $t$ in $R$.
Then, $F=\Q(\xi)$ is a 
quadratic number field, i.e.,
a field with $\dim_\Q(F)=2$.
Since the minimal polynomial of $\xi$ is a quadratic polynomial, $(R,\xi)$ is a $\Z[t]$-module with $t$ acting as multiplication by $\xi$.

For the remainder of this section 
we assume that $R=\Z[\xi]$ is the 
ring of integers in $F=\Q(\xi)$,
and hence that $R$ is a Dedekind domain.
We treat the non-Dedekind case in the next section.

Recall that an ideal $I$ of $R$ is
{\it invertible} if there is an ideal
$J$ such that $IJ\cong R$ as modules.

\begin{definition}\label{def:Pic}
The {\it Picard group} $\Pic(R)$ 
of a domain $R$
is the set of isomorphism classes of invertible ideals in $R$.
In this group, the product of 
$[I]$ and $[J]$
is the class of $[IJ]$.
\end{definition}

If $R$ is a Dedekind domain, 
every nonzero ideal is invertible,
and $\Pic(R)$ is the set of isomorphism classes of nonzero ideals in $R$.
We refer the reader to \cite[I.3]{WK} for basic facts about Dedekind domains, such as the fact that torsionfree $R$-modules are completely classified by their rank and their class in $\Pic(R)$.
In particular, $R$-modules isomorphic to $\Z^2$ as an abelian group have rank~1.
We refer the reader to \cite[Section 5]{lenstra} and \cite[Chapter 5]{cohen} for discussions on algorithms for computing $\Pic(R)$.

The group $\Pic(R[\xi^{-1}])$ is the quotient of $\Pic(R)$ by the subgroup 
generated by the prime ideals of $R$ dividing $\xi$; see \cite[Ex.I.3.8]{WK}.
If all these prime ideals are principal, $\Pic(R)\cong\Pic(R[\xi^{-1}])$.

Since every nonzero ideal
$I$ of $R$ has $\Z^2$ as its underlying abelian group,
each $(I,\xi)$ has the same minimal polynomial as $(R,\xi)$.
This proves:

\begin{theorem}\label{thm:Pic}
Let $R=\Z[\xi]$ be the ring of integers in a quadratic number field 
$\Q(\xi)$, with $\chi(t)$ the minimal polynomial of $\xi$. Then:
\begin{enumerate}
\item the  elements of $\Pic(R)$ are in 1--1 correspondence with the 
isomorphism classes of $\Z[t]$-modules $(\Z^2,T)$ with $\chi(T)=0$, 
with $T$ acting as $\xi$.
The Picard class of an ideal $I$ of $R$ corresponds to $(I,\xi)$.
\item the elements of~ $\Pic(R[\xi^{-1}])$~
are in 1--1 correspondence with the
shift equivalence classes of
matrices $T\in M_2(\Z)$
with $\chi(T)=0$. 

\end{enumerate}
In particular, if every prime ideal of $R$ dividing $\xi$ is principal, then
shift equivalence is the same as isomorphism for ideals of $R$.
\end{theorem}

\begin{corollary}
If  $d\not\equiv1\pmod4$ and $\Pic(\Z[\sqrt{d},1/\sqrt{d}])=0$, then the shift equivalence class of a matrix $T\in M_2(\Z)$ with 
\edit{tweaked}
$\chi(T) = t^2-d$ is determined by the rational canonical form of $T$.
\end{corollary}

\begin{remark}
In more concrete terms, two $2\times2$ 
matrices $T_1$, $T_2$ with the same
characteristic polynomial $\chi(t)$
determine ideals $I_1$, $I_2$ in $R=\Z[t]/(\chi)$ that are well defined up to isomorphism. 
Then $T_1$ and $T_2$ are shift equivalent if and only if $I_1[t^{-1}]$ and $I_2[t^{-1}]$ are isomorphic as $R[t^{-1}]$-modules.
\end{remark}

Suppose that $d$ is a nonzero integer with $|d|$ \emph{square-free},
and consider the ring of integers in $F=\Q(\sqrt{d})$.  
There are two cases:

\begin{description}
\item[Case 1:] If $d\not\equiv1\pmod4$, the ring of integers
in $\Q(\sqrt{d})$ is $R=\Z[\sqrt{d}]$. Letting $t$ 
act as $\xi=\sqrt{d}$,
we see from Theorem \ref{thm:Pic} that 
shift equivalence classes $(\Z^2,T)$ with
characteristic polynomial $t^2-d$ are in 1--1 correspondence
with elements of $\Pic(R[1/\sqrt{d}])$.
\end{description}

\begin{example}\label{d-ne-1}
1) If $\chi(t)=t^2-d$ for $d=2,3,6,7,11,14,19$ or $d=-1,-2,-7$, then $R=\Z[\sqrt{d}]$ and $\Pic(R)=0$ \cite{Ribenboim}\footnote{Alternatively, the reader may determine the order of the Picard group using the command {\bf NumberFieldClassNumber[$\sqrt{d}$]} in Mathematica\cite{numberfield}}.
For these values of $d$, there is only one shift equivalence class on $(\Z^2,T)$ with $\chi(t)=t^2-d$, namely the class of $T=\left(\mycom01\mycom{d}0 \right)$; $(\Z^2,T)$ is $(R,\sqrt{d})$.

2) If $\chi(t)=t^2+5$, then $\Pic(R)=\Z/2=\{R,I\}$, where
$R=\Z[\sqrt{-5}]$ and
$I=(2,1+\sqrt{-5})R$.
Since $\sqrt{-5}\not\in I$, it follows that 
$\Pic(R[(\sqrt{-5})^{-1}])=\Z/2$ as well. Thus
there are two non-isomorphic shift equivalence classes on $\Z^2$ 
with characteristic polynomial 
$t^2\!+5$\,: $R$ and $I$.
The matrices for $T$
corresponding to the bases
$\{1,\sqrt{-5}\}$ and $\{2,\sqrt{-5} \}$ are 
\[
\begin{pmatrix}
    0 & -5 \\ 1 & 0
\end{pmatrix}
\quad\text{and}\quad
\begin{pmatrix}
    -1 & -3 \\ 2 & 1
\end{pmatrix}.
\]

3) If $d=-6$ or $d=-10$, $\Pic(\Z[\sqrt{d}])\cong\Z/2$
but $\Pic(\Z[\sqrt{d},1/\sqrt{d}])=0$.
(See \cite[p.\,636]{Ribenboim}.)
In this case, the ideal 
$I=(2,\sqrt{d})$ is not
isomorphic to $R$, but the modules $R$ and $I$ 
are shift equivalent. The corresponding 
shift equivalent matrices are 
\[
\begin{pmatrix}
    0 & d \\ 1 & 0
\end{pmatrix}
\quad\text{and}\quad
\begin{pmatrix}
    0 & d/2 \\ 2 & 0
\end{pmatrix}.
\]
\end{example}

\begin{description}
\item[Case 2:]  
If $d\equiv1\pmod4$, the ring of integers in $\Q(\sqrt{d})$ is not $\Z[\sqrt{d}]$ but 
$\barR=\Z[\omega]$, where $\omega=\frac{1+\sqrt{d}}{2}$. 
We let $t$ act as $\xi=\omega$.
The minimal polynomial of $\omega$ is $t^2-t-c$,
where 
$c=(d-1)/4$.
\end{description}

By Theorem \ref{thm:Pic}, the isomorphism and shift equivalence classes 
$(\Z^2,T)$ with characteristic polynomial $t^2-t-c$ are in 
1--1 correspondence with elements of $\Pic(\barR)$ and
$\Pic(\barR[1/\omega])$, respectively.

\begin{example}\label{Pic=0}
If $d$ is $5,13,17,21,29$ or $-3,-7,-11,-19$ then 
$\barR=\Z[\omega]$ has $\Pic(\barR)=0$, and hence 
$\Pic(\barR[1/\omega])=0$,
so there is only one shift
equivalence class with characteristic polynomial $t^2-t-c$,
that of $\barR$, i.e., $T=(\mycom01\mycom{-c}1)$,
where $c=(d-1)/4$. 
\end{example}

\begin{remark}
For irreducible polynomials of degree $\ge3$, much less is known.
For example, little is known about
$R=\Z[t]/(\chi)$ when $\chi(t)$ is
$t^n+5t+10$ (a polynomial which is irreducible by Eisenstein's criterion).  In general, the computation of $\Pic(R)$ becomes unwieldy when $n$ gets bigger.
\end{remark}

\section{non-Dedekind subrings of number fields}
\label{sec:singular}

When $T$ is a $2\times2$ matrix, and its 
characteristic polynomial $\chi(t)$ is 
irreducible, the ring 
$R=\Z[t]/(\chi)$ 
is usually not integrally closed;
it is the integral closure $\barR$ of $R$ 
that is Dedekind \cite{AM}.
Recall \cite{AM} that an $R$-module 
$N$ is \emph{invertible} if there 
exists an $R$-module $N'$ 
such that $N \otimes_R N' \cong R$.
If $R/(\chi)$ is not integrally closed, 
not every $R$-module isomorphic 
to $\Z^2$ is invertible.
(For example, $\barR$ is not an invertible $R$-module.)

In this case, we need to supplement
the Picard group $\Pic(\barR)$ 
in Theorem \ref{thm:Pic} with another
invariant: the {\it conductor ideal.}
It is defined as $\cond=\textrm{ann}_R(\barR/R) =
\{r\in R\mid r\bar{R} \subseteq R\}$,
and is the largest ideal of $\barR$ contained in $R$.

Let $M$ and $M'$ be $R$-submodules of $\barR$. 
Since $\barR$ is $\Z^2$ as an abelian group, $M$ and $M'$ are also $\Z^2$ as  abelian groups. 
We want invariants to decide whether $(M,t)$ and $(M',t)$ are shift equivalent.

One invariant is the shift equivalence class of $(M\otimes_R\barR,t)$.
Since $M\otimes_R\barR$ is a rank~1 $\barR$-module, it is isomorphic to an ideal $I$ of $\barR$; the isomorphism $\phi:M\otimes_R\barR \smap{\cong}I$ is well defined up to multiplication by a unit of $\barR$.
%
Hence one invariant of $(M,t)$ is 
the shift equivalence class of
$(I,t)$ over $\barR$. Given $I$, and an isomorphism
$\phi:M\otimes_R\barR \smap{\cong}I$, 
we now show that the class of
$\barM=M/\cond I$ yields another invariant.
Since we can reconstruct $M$ from this data,
we get a classification of the $R$-modules
isomorphic to $\Z^2$.

\begin{theorem}\label{thm:M<I}
If $M$ is an $R$-module isomorphic to $\Z^2$  as an abelian group, and
$\phi:M\otimes_R\barR\smap{\cong}I$ is given,
there are canonical $R$-module inclusions 
$\cond I\subseteq M\subseteq I$.
Hence the $R$-modules isomorphic to $\Z^2$ are classified up to isomorphism by
\begin{enumerate}
\item the elements $[I]$ of~ $\Pic(\barR)$, and 
\item for each $[I]$, the equivalence classes of nonzero $R$-submodules 
$\barM = M/\cond I$ of $I/\cond I\cong \barR/\cond$, where $\barM\simeq\barN$ if 
$r\barM=\barN$ or $r\barN=\barM$
for some element $r$ of $\barR$.
\end{enumerate}
\end{theorem} 

\begin{proof}
Consider the short exact sequence 
$0\to R\to \barR \to \barR/R \to 0$.
Tensoring with $M$  yields the exact sequence
\[
\Tor^R_1(M,\barR) \to
\Tor^R_1(M,\barR/R)\map{\partial} M\oo_R R \to M\oo_R\barR \to M\oo_R(\barR/R)\to 0.
\]
There is a canonical isomorphism 
$M\cong M\oo_R R$, and
the term $M\oo_R\barR$
is isomorphic 
to $I$ by $\phi$. 
Since $M$ is a torsionfree 
abelian group, and the $\Tor$-module is torsion,
the map $\partial$ is zero. This gives the inclusion $M\subseteq I$.

Similarly, beginning with the short exact sequence $0\to \cond \to R\to R/\cond \to 0$ and tensoring with $M$,
the same argument yields the
assertion $\cond I\subseteq M$, since 
\[
M\oo_R\cond \cong 
M\oo_R({\barR}\oo_{\barR}\cond)\cong
(M\oo_R{\barR})\oo_{\barR}\cond
\cong I\oo_{\barR} \cond 
\map{\cong} \cond I.
\]
This construction depends
on the choice of isomorphism
$\phi:M\otimes\barR\smap{\cong} I$.
If $N=rM$ for nonzero $r\in\barR$, then
$N\cong M$ but $\phi(N)=r\phi(M)$.
Since $\Hom_{\barR}(I,I)=\barR$,
the choices of $\phi$ determine the
$R/\cond$-module up to multiplication by an element of $\barR$.
%
\end{proof}

\begin{remark}
    \label{rem:countingSEC}
    Theorem~\ref{thm:M<I} provides us with a simple count of an upper bound on the number of shift equivalence classes, namely, the product of the order of $\Pic(\barR)$, which is readily computable \cite{numberfield}, times the number of isomorphism classes of $R$-modules $M$ with 
$\cond\subseteq M\subseteq \barR$, which by Proposition~\ref{p:sqrt-d} is 
at  most four. 
Corollary \ref{R-ne-barR}
indicates that it is at least two.
\end{remark}

Our next family of examples concerns $T$ with $T^2=dI$, i.e., modules over 
$R=\Z[t]/(t^2-d)$ with
$T$ acting as $\sqrt{d}$.

\begin{example}\label{groupring}
($t^2=1$).  If $R=\Z[t]/(t^2-1)$, then
$\barR=\Z\times\Z$ and the conductor is $2\barR$. 
Theorem \ref{thm:M<I} applies
and says that the equivalence
classes correspond to the
equivalence classes of 
the four subgroups of
$\barR/2=\Z/2\times\Z/2$,
with $\barR/2$ corresponding to $\barR$ and 
the subgroup generated by 
$(1,1)$ corresponding to $R$.
The subgroups generated by
$(0,1)$ and $(1,0)$ correspond to
the $R$-modules $\Z\times2\Z$
and to $2\Z\times\Z$ of $\barR$,
both isomorphic to $\barR=\Z\times\Z$.
Hence there are only two shift equivalence classes, corresponding to $R$ and $\barR$. This provides an alternate calculation to
Example \ref{ex:permute}.
\end{example}

\begin{example}[$t^2=-4$]\label{ex:2i}
In this case $|d|$ is not 
square-free, so this does not fall under Case 1 of Section~\ref{sec:Q}.
Here $R=\Z[2i]$, and $\xi=2i$;
$\barR=\Z[i]$, 
$\Pic(\barR)=0$,   $\cond=2\barR$ and $\bar{R}/\cond \cong \Z/2 \times \Z/2$.
Because there are  4 nonzero subgroups of $R/\cond$,
there are three isomorphism classes of $R$-modules with 
$M\oo_R\barR\cong\barR$, namely 
$R\cong iR$, $J_1=(2,1+i)R$ 
and $\barR$.
(See below for why $R$
and $J_1$ are not isomorphic.)
Relative to the $\Z$-bases 
$\{1,2i\}$, $\{2,1+i\}$, and $\{1,i\}$ of these $R$-modules, 
$t=2i$ is represented by the matrices 
\[
\begin{pmatrix}
    0 & 1 \\ -4 & 0
\end{pmatrix},
\quad
\begin{pmatrix}
    -2 & 4 \\ -2 & 2
\end{pmatrix}
\quad\text{and}\quad
\begin{pmatrix}
    0 & 2 \\ -2 & 0
\end{pmatrix}.
\]
As $R[t^{-1}]=\Z[1/2,i]
=\barR[t^{-1}]$, 
these matrices are all shift equivalent.

In contrast $t=1+2i$ is represented on $R$, $J_1$ 
and $\barR$ by the respective matrices
\[
\begin{pmatrix}
    1 & 1 \\ -4 & 1
\end{pmatrix}
\quad
\begin{pmatrix}
    -1 & 2 \\ -2 & 3
\end{pmatrix}
\quad\text{and}\quad
\begin{pmatrix}
    1 & 2 \\ -2 & 1
\end{pmatrix}.
\]

In contrast $t=1+2i$ is represented on $R$, $J_1$ 
and $\barR$ by the respective matrices
\[
\begin{pmatrix}
    1 & 1 \\ -4 & 1
\end{pmatrix}
\quad
\begin{pmatrix}
    -1 & 2 \\ -2 & 3
\end{pmatrix}
\quad\text{and}\quad
\begin{pmatrix}
    1 & 2 \\ -2 & 1
\end{pmatrix}.
\]

These three matrices are in distinct shift equivalent classes,
even though they have the same canonical form 
and characteristic polynomial $t^2-2t+5$.

To see why $R\not\cong J_1$,
suppose that $f:R\to J_1$ has
$f(1)=2x+(1+i)y$, so
$f(2i)=-2(2x+2y)+(1+i)(4x+2y)$.
The map $f$ is represented by
the matrix
\[
A = \begin{bmatrix}
    x & -(2x+2y) \\ y & 4x+2y
\end{bmatrix},
\]
and $\det A = 4x^2 +4xy +2y^2 \neq \pm1$. Hence $f$ cannot be an isomorphism.
\end{example}

When $d\equiv1\pmod4$, $d\ne1$, then the 
integral closure of
$R=\Z[\sqrt{d}]$ is $\barR=\Z[\omega]$, $\omega=\frac{1+\sqrt{d}}{2}$.
It is convenient to use the
parameter $c=(d-1)/4$, as
$\omega^2-\omega-c=0$.

\begin{proposition} [$t^2=d$]
    \label{p:sqrt-d}
When $d\equiv1\pmod4$, $d\ne1$,
there are up to four 
isomorphism classes of 
$R$-modules $M$ with 
$\cond\subseteq M\subseteq \barR$,
namely: 
$R$, $J_0=(2,\omega) R$,  
$J_1=(2,1+\omega)R$, 
and $\barR$.  (Modulo $\cond$, 
these are the nonzero linear subspaces of $\barR/\cond$.)
Relative to the $\Z$-bases
$\{1,\sqrt{d}\}$, $\{2,\omega\}$, $\{2,1+\omega\}$ and $\{1,\omega\}$
of $R$, $J_0$, 
$J_1$ and $\barR$,
multiplication by $t=\sqrt{d}$
is represented by the matrices 
\[
\biggl(\mycom{0}{d}\mycom{1}{0}\biggr),\quad   
\biggl(\mycom{-1}{c}\mycom{4}{1}\biggr),\quad 
\biggl(\mycom{-3}4\,\mycom{c-2}{3}\biggr),\quad 
\textrm{and}\quad\biggl(\mycom{-1}{2c}\mycom21\biggr).
\]

Since $t=\sqrt{d}$ is relatively
prime to $\cond$ in $R$, 
the non-isomorphic $R$-modules among them
remain non-isomorphic modules over $R[t^{-1}]=R[1/\sqrt{d}]$.
That is, they are not shift equivalent.
\end{proposition}

\begin{proof}
The conductor ideal is 
$\cond=2\bar{R} =(2,1+\sqrt{d})$, and $|\barR/\cond|=4$. 
(If $c$ is even, $\barR/\cond=\F_2\times\F_2$, where $\F_2$ is the field of order~2;
if $c$ is odd then $\barR/\cond$ is the field $\F_4$ of order~4.)

By Theorem \ref{thm:M<I}, 
there are up to four 
isomorphism classes of 
$R$-modules $M$ with 
$\cond\subseteq M\subseteq \barR$,
namely: 
$R$, $J_0$, $J_1$, and $\barR$.
Modulo $\cond$, 
these are the nonzero linear subspaces of $\barR/\cond$.
(If $c$ is even, $\barR/\cond=\F_2\times\F_2$, 
where $\F_2$ is the field of order~2;
if $c$ is odd then $\barR/\cond$ 
is the field $\F_4$ of order~4.)
\end{proof}

\begin{corollary}\label{R-ne-barR}
If $d\equiv1\pmod4$, $d\ne1$,  then
$\barR$ is not  isomorphic to 
$R$, $J_0$ or $J_1$.
Therefore there are at least two shift equivalence classes. 
\end{corollary}

\begin{proof}
$\barR/\cond$ has 4 elements, while $R/\cond$, $J_0/\cond$ and
$J_1/\cond$ have only 2 elements.
Therefore, $\barR$ cannot be isomorphic to $R$, $J_0$ or $J_1$.
The conclusion follows from Proposition~\ref{p:sqrt-d}.
\end{proof}





Here is a simple example, 
showing how to apply 
Proposition \ref{p:sqrt-d}.
In the next section, we develop tools to apply 
Proposition \ref{p:sqrt-d}
more generally.

\begin{example}\label{d=5}
($t^2=5, c=1$).  
The ring of integers in
$\Q(\sqrt5)$ is
$\barR=\Z[\omega]$,
$\omega=\frac{1+\sqrt5}2$
is the fundamental unit,
and $\Pic(\barR)=0$.
In this case, $R$, $\omega R\cong J_0$ and
$\omega^2R\cong J_1$
are all isomorphic as
$R$-modules. 
By Theorem \ref{thm:M<I}
and Proposition \ref{p:sqrt-d},
there are exactly two
shift equivalence classes
with characteristic polynomial $t^2-5$.
They are represented by
the matrices
$\left(
\mycom 05\mycom 10
\right)$ and
$\left(
\mycom{-1}2\mycom 21
\right)$ (for the
$R$-modules $R$ and $\barR$);
the matrices
$\left(
\mycom{-1}1\mycom 41
\right)$ and
$\left(
\mycom{-3}4\mycom{-1}3
\right)$
are both shift equivalent to 
$\left(
\mycom 05\mycom 10
\right)$.

\end{example}

\section{Finding Isomorphisms}\label{sec:iso}

The first step towards exploiting Proposition~\ref{p:sqrt-d} 
and Corollary \ref{R-ne-barR}
is to identify whether or not 
$R$, $J_0$, and $J_1$ are 
isomorphic $R$-modules, as a function of 
the nonzero integer $c$.
\begin{enumerate}
\item Given the $\Z$-bases of Proposition~\ref{p:sqrt-d}, any $R$-module map $f \colon R\to J_0$ is determined by $f(1)=2x+\omega y$, because 
 $f(\sqrt{d}) =\sqrt{d}\cdot f(1) = 2(-x+yc)+\omega(4x+y)$ in $J_0$.
 The map $f$ is represented by $A\in M_2(\Z)$ where
\[
A = \begin{bmatrix}
    x & -x+cy \\ y & 4x+y
\end{bmatrix},
\]
which is an isomorphism if and only if the quadratic form 
\begin{equation}
    \label{eq:RJ0}
\det(A) =  Q(x,y) =
4x^2 + 2xy - cy^2 =\pm 1
\end{equation}
 has a solution  over $\Z$. 
 That is, $R$ and $J_0$ are isomorphic 
 $R$-modules if and only if $Q(x,y)=\pm1$ has a solution
 over $\Z$.



\item Similarly, a map 
$f \colon R\to J_1$ is determined by 
$f(1) = 2x + (1+\omega)y$ and 
\[ f(\sqrt{d}) =
\sqrt{d}f(1) = [(c-2)y - 3x]2 + (4x+3y)(1+\omega).
\]
Thus it is represented by $A\in M_2(\Z)$ where
\[
A = \begin{bmatrix}
    x & -3x+(c-2)y \\ y & 4x+3y
\end{bmatrix}.
\]
Thus $f$ is an isomorphism 
if and only if $(x,y)$ is a
solution to the quadratic form 
\begin{equation}
    \label{eq:RJ1}
\det(A) = Q(x,y) =  
4x^2 +6xy + (2-c)y^2 =\pm 1.
\end{equation}

\begin{remark}
\label{symmetric}
$R$ is isomorphic to $J_0$
if and only if
$R$ is isomorphic to $J_1$.
Indeed, a map $f_0:R\to J_0$
with $f_0(1)=2x+\omega y$ is an isomorphism
if and only if the map
$f_1:R\to J_1$ is an isomorphism, where
$f_1(1)=2x+(1-\omega)y$.
\end{remark}

\item We can use a similar computational scheme
to compare $J_0$ and $J_1$, using
the given bases of these $R$-modules.
 Set
\[
\begin{aligned}
    f_2(2) & = x\cdot 2 + y(1+\omega) \\
    f_2(\omega) & = u\cdot 2 + v(1+\omega)
\end{aligned}
\]
Then, 
 regarding $J_1$ as a subgroup of $\barR$, we have
\[
\begin{aligned}
f_2(2\omega) =\omega f_2(2) 
     & = x\cdot 2\omega + y(1+\omega)\omega \\ 
&= 2x\omega + y\omega + \frac{y}{4} (1 + 2\sqrt{d} + d) \\
& =2x\omega + y\omega + \frac{y}{4} + \frac{y}{2}\sqrt{d} + \frac{y}{4}(4c+1) \\
&= 2(x+y)\omega +cy \\
&= 2(x+y)(1+\omega) +cy - 2(x+y)\\
& = \left(\left(\frac{c}{2} - 1\right)y -x\right)\cdot 2 +2(x+y)\cdot(1+\omega)
\end{aligned}
\]
Therefore
\[
\begin{aligned}
    v & = x + y \\
    4u & = (c-2)y -2x.
\end{aligned}
\]
The $M_2(\Z)$ representation of $f_2$ is 
\[
A_2 = \begin{bmatrix}
    x & \frac{1}{4}\left((c-2)y -2x \right)\\ y & x+y
\end{bmatrix}.
\]
Observe that $u$ must be an integer which is equivalent to $(c-2)y -2x = 4k$ for some integer $k$.
For $f_2$ to be an isomorphism it must be the case that 
\[
\det(A_2) = Q(x,y) = x^2 +\frac{3}{2}xy - \frac{c-2}{4}y^2 =\pm 1,
\]
which is equivalent to solving
\[
4x^2 +6xy - (c-2)y^2 =\pm 4.
\]
Using the constraint that $2x = (c-2)y-4k$ we  conclude that $f_2$ is an isomorphism if and only there exist integers $k$ and $y$ that solve
\begin{equation}\label{eq:J12}
c(c-2)y^2 - 4(2c-1)ky + 16k^2 = \pm 4.
\end{equation}
\end{enumerate}

\begin{lemma}
\label{lem:cJJ}
    If $c$ is even or $c \leq -3$, then $R$ is not isomorphic to $J_0$ or $J_1$.
    If $c \leq -5$, then $J_0$ is not isomorphic to $J_1$.
    
If $c =-4$, then $J_0$ and $J_1$ are isomorphic.
It follows from
Proposition \ref{p:sqrt-d} that
there are 3 isomorphism classes 
of $R$-modules $M$ with
$M\otimes_R\barR\cong\barR.$
\end{lemma}

\begin{proof}
    The parity of \eqref{eq:RJ0} and \eqref{eq:RJ1} shows that if $c$ is even, then there cannot be any solutions.
    Applying Mathematica's
    {\bf FindInstance}
\cite{reference.wolfram_2022_findinstance} shows that if $c\leq -3$ then there are no solutions; in fact, the appropriate $Q$ are positive definite in these ranges.

When $c=-4$,
$(x,y)=(1,1)$, $k=2$, 
defines an isomorphism 
$J_0\cong J_1$.
\end{proof}

\begin{remark}\label{rem:J0J1}
Using Mathematica again,
we discover that 
if $c\ge9$ then
$J_0\not\cong J_1$,
because there are 
no solutions to
\eqref{eq:J12},
and $Q$ is positive definite in this range.
We also see that there
appear to be infinitely 
many values of $c$ for which $J_0\cong J_1$, and infinitely many values of $c<0$ for which 
$J_0\not\cong J_1$.
\end{remark}




\begin{example}[$t^2=-15$, $c=-4$]
\label{d=15}
The ring $\barR=\Z[\omega]$ of integers in $\Q(\sqrt{-15})$
has $\Pic(\barR)=\Z/2$ on the 
class of $I=(2,\omega)\barR$,
where $\omega=\frac{1+\sqrt{-15}}{2}$.
By Theorem \ref{thm:M<I},
Proposition \ref{p:sqrt-d},
Remark~\ref{symmetric}
and Lemma ~\ref{lem:cJJ},
$R\not\cong J_0$ and 
$J_0\not\cong J_1$.
Thus there are 6 non-isomorphic 
$R$-modules $M$ with
underlying group $\Z^2$: 3 with 
$M\otimes\barR\cong\barR$
and 3 more with $M\otimes\barR\cong I$.

Since $I[1/\omega]\cong \barR[1/\omega]$, we have $\Pic(\barR[1/\omega])=0$.
As in Proposition \ref{p:sqrt-d}, 
they represent the 4 distinct 
shift equivalence classes with $\chi(t)=t^2+15$.
\end{example}

\begin{remark}\label{units}
When $d<-3$, $\Q[\sqrt{d}]$ is an imaginary number field, and the
only units of $\Z[\omega]$ are $\pm1$. When $d>0$, there is a
``fundamental unit'' $\eta$ of infinite order, and every unit of
$\Z[\omega]$ is $\pm\eta^n$ for an integer $n$. Fundamental units can be found using the Mathematica command
{\bf NumberFieldFundamentalUnits.}
\end{remark}
\goodbreak

\begin{example}[$t^2=101, c=25$]
\label{d=101}
The ring of integers in $\Q(\sqrt{101}])$ is
$\barR=\Z[\omega]$, where $\omega=\frac{1+\sqrt{101}}{2}$ and  
$\omega^2-\omega-25=0$.
Now $\Pic(\barR)=0$, and the fundamental unit is 
$\eta=10+\sqrt{101}$.

Set $R=\Z[\sqrt{101}]$, and note that $\eta\in R$.
By Theorem \ref{thm:M<I},
Corollary \ref{R-ne-barR} 
and Remark \ref{rem:J0J1},
there are 4 isomorphism classes of 
$R$-modules with underlying group $\Z^2$.
Hence there are  4 shift equivalence classes
of $R$-modules with 
$t=\sqrt{101}$.
For the $\Z$-bases of 
Proposition \ref{p:sqrt-d},
the matrices are 
\[
\begin{pmatrix}
    0 & 1 \\ 101 & 0
\end{pmatrix},\quad
\begin{pmatrix}
    -1 & 4 \\ 25 & 1
\end{pmatrix},\quad
\begin{pmatrix}
    -3 & 23 \\ 4 & 3
\end{pmatrix},
\quad \text{and}\quad
\begin{pmatrix}
    -1 & 2 \\ 50 & 1
\end{pmatrix}.
\]

For every monic irreducible quadratic polynomial $f$ with roots $r,\bar r \in
R$, there are 4 isomorphism classes of matrices 
$T$ acting as $r$.
If $r$ is prime to 
$\cond=(2,\sqrt{101})R$, 
these matrices will have 4 distinct shift equivalence classes.
For example,  $\eta=10+\sqrt{101}$ is a root of
the polynomial $f(t)=t^2-20t-1$.
Hence there are 4 distinct shift equivalence classes of $\Z[t]$-modules
$\Z^2$ with $t=\eta$.
\end{example}

We can now recover a well-known result; 
see \cite[p.81]{PT}. 
Set $J_0=(2,\omega) R$.
\begin{lemma}
\label{ParryTuncel}
The matrix $T=\left( \mycom{19}4 \mycom51 \right)$
is not shift equivalent to
its transpose $T^t=\left( \mycom{19}5 \mycom41\right)$.
\end{lemma}

\begin{proof}
The matrix $T$ represents
$t=\eta$ acting on the basis
$\{5,-9+\sqrt{101}\}$ of 
$J_0$, and
$T^t$ is the matrix of
$t=\eta$ acting on the basis
$\{2,-2+\sqrt{101}\}$ of $R$.
By Theorem \ref{thm:M<I} and 
Example \ref{d=101}, there are
4 shift equivalent classes of $R$-modules with $t=\eta$.
Since $R\not\cong J_0$,
the $\Z[t]$-modules 
$J_0$ 
and $R$ are not shift equivalent.
\end{proof}


When $\chi(t)$ is a polynomial of degree more than 2, the computational difficulty explodes. We give a simple example, with 3 Jordan blocks 
over $\C$, to illustrate some of the techniques involved.

\begin{example}
\label{ex:3perm}
Consider the case
$\chi(t)=t^3-1$, which is the characteristic polynomial of both $T$ (the rotation matrix),
as well as $T_2$ and $T_3$:
\[ T =
\begin{pmatrix}
0 & 0 & 1\\ 1 & 0 & 0\\0 & 1 & 0
\end{pmatrix} 
\quad\textrm{and}\qquad T_2=
\begin{pmatrix}
1& 0 & 0\\0&0&-1\\0 & 1 & -1
\end{pmatrix}
\quad\textrm{and}\qquad T_3=
\begin{pmatrix}
1& 1 & 0\\1&-2&-2\\0 & -2 & 1
\end{pmatrix}.
\]
The integral closure of the ring $R=\Z[t]/(\chi)$ is
$\barR=\Z\times\Z[\omega]$,
where $\omega=\root3\of1$; 
the map $R\to\barR$ 
sends $t$ to $(1,\omega)$.
Note that $\cond=(3,\omega-1)\barR$. 
Since $R/\cond=\F_3$
and $\barR/\cond\cong\F_3\times\F_3$,
the $R$-modules isomorphic to $\Z^2$
correspond to the 5 nonzero 
$\F_3$-subspaces of
$V=\F_3\times\F_3$.

There are $3$ shift equivalence 
classes with $\chi(t)=t^3-1.$ 
In more detail, 
Theorem \ref{thm:M<I} 
implies that 
$V=\F_3\times\F_3$ corresponds 
to $\barR$, with matrix $T_2$
(for the basis $\{(3,0),(0,1),(0,\omega)\}$),
and the diagonal subspace 
on $(1,1)$ corresponds to $R$,
with matrix $T$ (for the basis 
$\{1,t,t^2\}$).
The two 1-dimensional subspaces 
of $V$, $\F_3(1,0)$ and $\F_3(0,1)$
correspond to the $R$-modules
$\Z\times(\omega-1)\Z[\omega]$ and
$3\Z\times\Z[\omega]$, both 
isomorphic to $\barR$.
The final 1-dimensional subspace
$\F_3(1,2)$ of $V$ corresponds to the
$R$-submodule $M$ with basis
$\{(2,1),(0,\omega-1),(1,2)\}$;
the associated matrix is $T_3$.
\end{example}

\section{Shift equivalence over $\Z/p^n$}
\label{sec:Z/pn}

We briefly consider shift equivalence of $(M,T)$ when $M$ is a finite $p$-group, i.e., shift equivalence over $R=\Z/p^n$.
The classification of finite Artinian modules over $\Z/p^n[t]$ for all $n$ is equivalent to the classification of finite ($p$-primary) Artinian modules over $\Z_p[t]$, where $\Z_p$ is the $p$-adic integers \cite{AM}. 
The associated primes over these modules contain $p$ and are in 1--1 correspondence with the prime ideals in $\Z/p[t]$, such as $(p,t-\lambda)$. 
We can ignore the subgroup $M_{\mathrm{nil}}$ on which $t$ acts nilpotently, as $M_{\mathrm{nil}}$ is shift equivalent to $0$ and $M$ is shift equivalent to $M/M_{\mathrm{nil}}$; see Lemma \ref{lem:no-t}. 
We therefore restrict to the case when $t$ is an automorphism of $M$.

We do not know of a complete set of invariants for shift equivalence in this setting.
A partial list can be obtained by observing that 
$M$ determines $\barM_j := M/p^jM$, so that
shift equivalence of $\barM_j$ for $j=1,\ldots, n$
gives a family of invariants.
To give a sense of the relevant calculations we note that $\barM_1=M/pM$, so the rational canonical form
of $T$ mod $p$ is an invariant of 
the shift equivalence class of $M$.

Suppose that $M$ is $(\Z/p^n)^2$, so $T$ is a 
$2\times2$ matrix over $\Z/p^n$, with 
characteristic polynomial $\chi(t)$.
Thus there is either 
one block (and $M/pM\cong\F_{p^2}$)
or two 1-dimensional blocks (and $M/pM\cong \F_p^2$).
The analysis is governed by the considerations
in Section \ref{sec:Z}.

\goodbreak
\begin{example}\label{Z/pn}
Suppose $M_a$ is $(\Z/p^n)^2$  with 
$T=\left(\mycom{\lambda}a\mycom0{\lambda}\right)$ 
for some $a\in\Z/p^n$, and $\lambda$ is not nilpotent 
(i.e., not divisible by $p$).
Since every element of $\Z/p^n$ is 
either a unit or nilpotent, 
$\lambda$ must be a unit of $\Z/p^n$.
As in Proposition \ref{one-lambda},
we see from  \eqref{eq:fmap} that
 $M_a$ and $M_b$ are shift equivalent
if and only if $as=br$ for units $r,s$, i.e., 
$a$ and $b$ differ by a unit of $\Z/p^n$.

Since each nonzero $a\in\Z/p^n$ is
$up^k$ for a unit $u$ and a 
unique $k$, $0\le k\le n-1$,
every $M_a$ is shift equivalent to 
exactly one of $M_0, M_1, M_p, M_{p^2},...,M_{p^{n-1}}$.
\end{example}

Arguments similar to those employed in Example~\ref{Z/pn} apply to the general case when $\chi(t)$ factors as $(t-\lambda_1)(t-\lambda_2)$, where  $\lambda_1 \ne\lambda_2$ 
are elements of $\Z/p^n$.
As in Proposition \ref{Ma=Mb}, 
the classification of shift equivalence 
classes is more complicated, as it depends on $\lambda_1-\lambda_2$.
Again, returning to \eqref{eq:fmap} we see that $M_a$ and $M_b$ are shift equivalent if and only if $br-as =u(\lambda_1 -\lambda_2)$ for units $r,s\in \Z/p^n$.
Thus, for example if $(\lambda_1 -\lambda_2)=1$, then there is a unique shift equivalence class since one is free to choose $u=br-as$.

\medskip
\goodbreak
We conclude our cautionary tale with
a peek into the jungle of modules
over $\Z/p^3$.
Consider the following quotient ring of $\Z_p[t]$:
\[
R_\lambda=\Z_p[t]/(p^3,(t-\lambda)^2,p^2(t-\lambda)).
\]
By \cite[6.1]{KL}\cite[3.2]{KWW}, 
$R_\lambda$ is 
``finite-length wild'': 
any description of finite $R_\lambda$-modules would have to contain a description of all
finite-dimensional modules 
over finite $\Z/p$-algebras.  This is
generally considered to be hopeless,
in the sense that it is an impractically complicated computational task. This notion of wildness
goes back to \cite{GP}.

\begin{example} Consider 
$M=(\Z/p^3)\oplus(\Z/p^2)$ with
$t(x,y)=((\lambda+up^2)x,\lambda y+px)$; 
$\barM=M_p$ does not recover $u$.
In fact, $M$ is a module over 
the ring $R_\lambda$.
\end{example}

\appendix
\section{Simple Mathematica Code}
The following Mathematica code provides information about the existence or nonexistence of isomorphisms between $R$, $J_0$, and $J_1$ as discussed in Section~\ref{sec:iso} for $100\leq c\leq 100$.
\medskip

\begin{verbatim}
For[c = -101, c < 99, c++;
  Print["c=", c, "  R iso J0 ", 
    !And[ResourceFunction["EmptyQ"][FindInstance[4x^2 + 2x*y - c*y^2 == 1, 
    {x, y}, Integers]], 
    ResourceFunction["EmptyQ"][FindInstance[4x^2 + 2x*y - c*y^2 == -1, 
    {x, y}, Integers]]], 
    "   R iso J1 ", 
    !And[ResourceFunction["EmptyQ"][FindInstance[4x^2+6x*y+(2-c)*y^2 == 1, 
    {x, y},Integers]], 
    ResourceFunction["EmptyQ"][FindInstance[4x^2+6x*y+(2-c)*y^2 == -1, 
    {x, y}, Integers]]], 
    "   J0 iso J1 ", 
    !And[ResourceFunction["EmptyQ"][FindInstance[c*(c-2)*y^2-4*(2*c-1)*k*y+16*k^2 == 4, 
    {k, y}, Integers]], 
    ResourceFunction["EmptyQ"][FindInstance[c*(c-2)*y^2-4*(2*c-1)*k*y+16*k^2 == -4, 
    {k, y}, Integers]]]
  	]
 	]
\end{verbatim}

\section*{Acknowledgments}
We would like to thank A. Kontorovich for his assistance
with solving the Diophantine equations discussed in Section~\ref{sec:iso}. 

\bibliographystyle{siamplain}
\bibliography{km, chuck}
\end{document}